\documentclass[preprint,12pt]{elsarticle}

\usepackage{graphicx}
\usepackage{booktabs}
\usepackage{bm}
\usepackage{combelow}
\usepackage[all,web]{xy}
\usepackage{amsthm}
\usepackage{amssymb}
\usepackage{mathtools}
\usepackage{xfrac}
\usepackage[T1]{fontenc}
\usepackage{textcomp}
\usepackage[scr=boondox, frak=boondox]{mathalpha}

\def\CC{\mathbb{C}}

\def\PP{\mathbb{P}}
\def\QQ{\mathbb{Q}}
\def\RR{\mathbb{R}}

\def\ZZ{\mathbb{Z}}
\def\hyper{\mathfrak{H}}

\newcommand{\mtwo}[4]{\left(
        \begin{matrix}#1&#2\\#3&#4
        \end{matrix}\right)}
\newcommand{\abcd}[4]{\left(
        \begin{smallmatrix}#1&#2\\#3&#4\end{smallmatrix}\right)}
\newcommand{\legendre}[2]{\genfrac{(}{)}{}{}{#1}{#2}}

\DeclarePairedDelimiter\floor{\lfloor}{\rfloor}	
\def\mathalfa@frakscaled{s*[1]}
\makeatother
\DeclareFontFamily{U}{BOONDOX-frak}{\skewchar \font =45}
\DeclareFontShape{U}{BOONDOX-frak}{m}{n}{
    <-> \mathalfa@frakscaled  BOONDOX-r-frak}{}
\DeclareFontShape{U}{BOONDOX-frak}{b}{n}{
    <-> \mathalfa@frakscaled  BOONDOX-b-frak}{}
\DeclareSymbolFont{boondoxfrak}{U}{BOONDOX-frak}{m}{n}
\SetSymbolFont{boondoxfrak}{bold}{U}{BOONDOX-frak}{b}{n}
\DeclareSymbolFontAlphabet{\mathfrak}{boondoxfrak}

\newtheorem{theorem}{Theorem}
\newtheorem{proposition}{Proposition}

\journal{Journal of Number Theory}

\begin{document}

\begin{frontmatter}

\title{On the Ramanujan-Ligozat modular unit} %% Article title

\author{Carlos Casta\~no-Bernard} %% Author name



\affiliation{organization={Universidad de Colima},
            addressline={Bernal Diaz del Castillo 340}, 
            city={Colima},
%          citysep={}, % Uncomment if no comma needed between city and postcode
            postcode={28045}, 
            state={Colima},
            country={Mexico}}

\begin{abstract}
Let $\chi$ be the quadratic Dirichlet character
of conductor $N$ associated with a given real quadratic field $k$
and let $X_\chi(N)$ be the double covering of $X_0(N)$ associated with $\chi$.
In this note, we study a modular unit $f$ on the curve  $X_\chi(N)$.
We prove that the transform of $f$ by the Fricke involution $\mathscr{w}_N$ is
the level $N$ analogue of the Rogers-Ramanujan continued fraction
introduced by Lee and Park.
Furthermore, if $N$ is a positive prime discriminant, then $f$ is Ligozat's modular unit,
which is known from Mazur's work on the Eisenstein ideal.
We extend Ligozat's construction to
any positive fundamental discriminant $N$.
\end{abstract}


\end{frontmatter}

%%%%%%%%%%%%%%%%%%%%%%%%%%%%%%%%%%%%%%%%%%%%%%%%%%%%%%%%%%%%%%%%%%%%%%%%
%
% Main text
%
%%%%%%%%%%%%%%%%%%%%%%%%%%%%%%%%%%%%%%%%%%%%%%%%%%%%%%%%%%%%%%%%%%%%%%%%
\section{Introduction}

\subsection{Statement of main results}\label{sub:shim}
Let $F$ be a real quadratic field and
let $\QQ(\zeta)$ be the minimal cyclotomic field that contains $F$
in a fixed algebraic closure of $\QQ$.
We have $F\subset\QQ(\zeta)^+ = \QQ(\zeta + \zeta^{-1})$ and,
assuming that $\zeta$ is a primitive $N$-\textit{th} root of unity,
we have a diagram of Galois field extensions
\begin{equation*}
\xymatrix{
 & \QQ(\zeta)^+ \ar@{-}[dd]^{\mathcal{G}} \ar@{-}[ld]_{\ker(\chi)} \\
k \ar@{-}[rd]_{\{1,\varepsilon\}} &  \\
 & \QQ 
}
\end{equation*}
where $\chi$ is the primitive quadratic Dirichlet character on the group $\mathcal{G} = (\ZZ/N\ZZ)^\times/\{-1,1\}$,
and $\sigma$ is the non-trivial automorphism of $F$.
Let $K_2$ be the Milnor functor
and $k_{2,F}$ be the order of the finite group $K_2\mathcal{O}_F$,
where $\mathcal{O}_F$ is the ring of integers of $F$.
As a special case of Birch and Tate's conjecture,
proven by Wiles~\cite[p. 499]{wiles:iwasawa},
we have
\begin{equation*}
k_{2,F} = B_{2,\chi},
\end{equation*}
where $B_{2,\chi}$ is the second generalised Bernoulli number that Leopoldt~\cite{leopoldt:bernoulli} associated with $\chi$.
Let us assume that $N\geq 12$.
Then the work of Fox, Urbanowicz, and Williams~\cite[p. 819]{urbanowicz:progress}
implies that $4\mid k_{2,F}$.
Thus  $4\mid  B_{2,\chi}$ and the product
\begin{equation*}
f = \prod_{1\leq r\leq \lfloor \frac{N-1}{2}\rfloor} g_r^{\chi(r)},
\end{equation*}
where $g_r$ denotes a Siegel unit associated with $r$,
defines a modular unit on the complete modular curve $X_1(N)$ over $\QQ$.
(See Subsection~\ref{sub:diamonds}.)
It is well-known that the group $\mathcal{G}$ acts faithfully on the curve $X_1(N)$ via the diamond automorphisms $\langle \cdot \rangle$.
Let $\chi^{\bm{\mu}}$ be the multiplicative M\"obius inverse of $\chi$ with respect to $N$.
(See Subsection~\ref{sub:schering}.)
This leads to the first main result.
%%%%%%%%%%%%%%%%%%%%%%%%%%%%%%%%%%%%%%%%%%%%%%%%%%
%
%       Theorem 1
%
%%%%%%%%%%%%%%%%%%%%%%%%%%%%%%%%%%%%%%%%%%%%%%%%%%
\begin{theorem}\label{thm:one}
As mentioned above, we consider a fundamental discriminant $N\geq 12$.
For each $d\in \mathcal{G}$ the action $f^{\langle d \rangle}$ of the diamond automorphism $\langle d \rangle$ on
the modular unit $f$ is given by
\begin{equation*}
f^{\langle d \rangle} = \chi^{\bm{\mu}}(d) f^{\chi(d)}.
\end{equation*}
\end{theorem}

To state our second main result,
let $h_F$ be the class number of the real quadratic field $F$, $\varepsilon_F>1$ be the fundamental unit of $F$,
and recall that the Fricke involution $\mathscr{w}$ of $X_\chi(N)$ exchanges cusps $0$ and $\infty$.
%%%%%%%%%%%%%%%%%%%%%%%%%%%%%%%%%%%%%%%%%%%%%%%%%%%%%%%%%%%%%%%%%%%%%%%%%%%%%%
%
%        Theorem 2
%
%%%%%%%%%%%%%%%%%%%%%%%%%%%%%%%%%%%%%%%%%%%%%%%%%%%%%%%%%%%%%%%%%%%%%%%%%%%%%%
\begin{theorem}\label{thm:two}
For each positive fundamental discriminant $N$ we have
\begin{equation*}
f = \varepsilon_F^{-h_F} + O(q)
\end{equation*}
and
\begin{equation*}
f^{\mathscr{w}} = q^{\tfrac{1}{4}B_{2,\chi}}\prod_{n=1}^\infty(1-q^n)^{\chi(n)},
\end{equation*}
where $q = e^{2\pi i\tau}$ and $\tau$ lies in the upper half-plane $\hyper$.
\end{theorem}

\subsection{Motivations}
Theorem~\ref{thm:one} was motivated by Mazur's comment~\cite[p. 107]{mazur:eisen}.
More precisely,
let $X_\chi(N)$ be the modular curve defined by the commutativity of the following diagram of Galois coverings
\begin{equation*}
\xymatrix{
 & X_1(N) \ar@{->}[dd]^{\mathcal{G}} \ar@{->}[ld]_{\ker(\chi)} \\
X_\chi(N) \ar@{->}[rd]_{\{1,\nu\}} &  \\
 & X_0(N)
}
\end{equation*}
where $\nu$ denotes the projection of $\langle d\rangle$ to $X_\chi(N)$,
for any $d\in\mathcal{G}$ such that $d\not\in\ker(\chi)$.
Theorem~\ref{thm:one},
together with the basic properties of the multiplicative M\"obius inversion formula,
implies that $f$ is a modular unit on the curve $X_\chi(N)$ over $\QQ$ and either
\begin{equation}\label{eqn:ol}
f^\nu = -\frac{1}{f},
\end{equation}
if the fundamental discriminant $N$ is prime or
\begin{equation*}
f^\nu = \frac{1}{f},
\end{equation*}
otherwise.
The assertion that
for each prime $N\equiv 1\pmod{4}$ greater than $5$
the function $f$ is a modular unit on the curve $X_\chi(N)$ that satisfies Equation~\ref{eqn:ol},
is a well-known result of Ligozat found in the work of Mazur~\cite[p. 107]{mazur:eisen},\footnote{
Only a part of Ligozat's unpublished manuscript is discussed,
and Mazur added that he hopes that Ligozat will
present the full story in his future
publications.}
which he uses to obtain a divisor on the curve $X_\chi(N)$ which
is the pullback of a
certain divisor on the curve $X_0(N)$ that gives the ``extra'' point on
the Jacobian $J_0(N)$.
We hope that Theorem~\ref{thm:one}, together with the divisibility properties of $B_{2, \chi}$,
for a primitive even quadratic character $\chi$
due to Urbanowicz~\cite{urbanowicz:b2},
and the work of Wake and Wang-Erickson~\cite{wwe:squarefree}
on the analogue of Mazur's Eisenstein ideal for
squarefree level $N$, may suggest some interesting arithmetical applications.

Theorem~\ref{thm:two}
generalises the $q\rightarrow 1$ limit formula of some well-known continued fraction identities.
Indeed, if $N = 8$ then % $\tfrac{1}{4}B_{2,\chi}=\frac{1}{2}$
the Ramanujan-G\"ollnitz-Gordon continued fraction can be expressed as
\begin{equation*}
f^\mathscr{w} =\cfrac{q^{1/2}}{1+ q+\cfrac{q^2}{1+ q^3+\cfrac{q^4}{1+ q^5+\cfrac{q^6}{1+q^7+\ddots}}}},
\end{equation*}
% which is a Hauptmodul on a covering of the curve $X_\chi(8)$.
% (See Cho, Koo, and Park~\cite[p. 932]{koo:goenitz}.)
and taking the limit of the right-hand side as $q\rightarrow 1$, we arrive at
\begin{equation*}
f = \frac{-2 + \sqrt{8}}{2} + O(q),
\end{equation*}
and if $N = 5$  then % $\tfrac{1}{4}B_{2,\chi}=\frac{1}{5}$
the Rogers-Ramanujan continued fraction may be expressed as
\begin{equation*}
f^\mathscr{w} =\cfrac{q^{1/5}}{1+\cfrac{q}{1+\cfrac{q^2}{1+\cfrac{q^3}{1+\ddots}}}},
\end{equation*}
% which is known to be a Hauptmodul of a covering of the curve $X_\chi(5)$.
and  taking the limit of the right-hand side as $q\rightarrow 1$, we arrive at
\begin{equation*}
f = \frac{-1 + \sqrt{5}}{2} + O(q).
\end{equation*}
Furthermore, Theorem~\ref{thm:one} and Theorem~\ref{thm:two} imply that $f^\mathscr{w} + f^{\nu\mathscr{w}}$
is a rational function on the modular curve $X_0(N)$,
which is a generalisation of Lee and Park~\cite[Theorem 1.1, p. 308]{leepark:level13} to higher levels.

%\section*{Acknowledgement}
%The author wishes to thank
%Gari Peralta-\'Alvarez
%for his helpful comments.
%%%%%%%%%%%%%%%%%%%%%%%%%%%%%%%%%%%%%%%%%%%%%%%%%%%%%%%%%%%%%%%%%%%%%%%%
%
%                   The proof
%
%%%%%%%%%%%%%%%%%%%%%%%%%%%%%%%%%%%%%%%%%%%%%%%%%%%%%%%%%%%%%%%%%%%%%%%%
\section{The proofs}

\subsection{The Schering symbol}\label{sub:schering}
%
% https://personal.math.ubc.ca/~cass/research/pdf/QuadraticCFT.pdf
%
% https://kconrad.math.uconn.edu/blurbs/gradnumthy/Gauss-Jacobi-sums.pdf
%
% http://alpha.math.uga.edu/~pete/wilson_easy.pdf
%
The Kronecker symbol associated to a positive integer $m$ and an element $a\in(\ZZ/m\ZZ)^\times$
is defined by
\begin{equation*}
\legendre{a}{m} = \prod_{i=1}^n \legendre{a}{p_i}^{e_i},
\end{equation*}
where $m = p_1^{e_1}\dots p_n^{e_n}$ is the prime factorisation of $m$, and $(\frac{a}{p})$ is
either the Legendre symbol, if $p$ is odd or
\begin{equation*}
\legendre{a}{p} = (-1)^{\frac{a^2 + 1}{8}},
\end{equation*}
if $p=2$.
Given a positive integer $m$,
we define the \textit{Schering symbol} $\legendre{\cdot}{m}^{\bm{\mu}}$ 
as the multiplicative M\"obius inverse of the Kronecker symbol $\legendre{\cdot}{m}$.
That is,
\begin{equation*}
\legendre{\cdot}{m}^{\bm{\mu}} = \prod_{d d^\prime = m} \legendre{\cdot}{d}^{\mu(d^\prime)}.
\end{equation*}
From well-known properties of the M\"obius inversion formula, we have
\begin{equation*}
\legendre{a}{m}^{\bm{\mu}} = 
\begin{cases}
    \legendre{a}{p} & \text{if $m=p^k$,} \\
        1         & \text{otherwise,}
    \end{cases}                                                                                                 
\end{equation*}
where $p$ is a prime number.

%%%%%%%%%%%%%%%%%%%%%%%%%%%%%%%%%%%%%%%%%%%%%%%%%%
%
% Tom Leinster, The Euler characteristic of a category
%
% https://arxiv.org/pdf/math/0610260
%
%  Gauss, Eisenstein, and the "Third" Proof of the Quadratic Reciprocity Theorem: Ein kleines Schauspie
% /home/carlos/Dropbox/Neue_Forschungsarbeit/laubenbacher
%
%%%%%%%%%%%%%%%%%%%%%%%%%%%%%%%%%%%%%%%%%%%%%%%%%%
\begin{proposition}\label{prop:key}
Let $N$ be a positive fundamental discriminant.
For each $a\in(\ZZ/N\ZZ)^\times$ the Schering symbol can be expressed as:
\begin{equation*}
\legendre{a}{N}^{\bm{\mu}} = (-1)^{ \mathscr{s}(a, N) + (a + 1)\mathscr{h}_N},
\end{equation*}
where
\begin{equation*}
\mathscr{s}(a, N) = \sum_{x\in H}\floor*{\frac{ax}{N}}
\end{equation*}
and
\begin{equation*}
\mathscr{h}_N=\sum_{x\in H}x,
\end{equation*}
with summation over the half-system $H$ consisting of
the integers $1\leq x \leq N^\prime$ such that $(x, N) = 1$,
and
\begin{equation*}
N^\prime = 
\begin{cases}
    \tfrac{N}{2} -1          & \text{if $N$ is even,} \\
    \tfrac{N - 1}{2}         & \text{if $N$ is odd.}
    \end{cases}                                                                                                 
\end{equation*}
\end{proposition}

\begin{proof}
Let $\bar a\in (\ZZ/N\ZZ)^\times$.
We divide into cases according to the parity of $N$:

\bigskip
\noindent\textit{Case 1: Assume $N$ is odd}. 
There is a permutation $\pi\in\textrm{Perm}(H)$ and a mapping
\begin{equation*}
\varepsilon : H \longrightarrow \{-1,1\}
\end{equation*}
such that for each $x\in H$:
\begin{equation*}
ax \equiv \varepsilon(x)\pi(x) \pmod{N}.
\end{equation*}
Thus
\begin{displaymath}
ax=\left\{
\begin{array}{ll}
\floor*{\frac{ax}{N}}N + \varepsilon(x)\pi(x), &\textrm{if $\varepsilon(x) = 1$},\\ 
&\\
\floor*{\frac{ax}{N}}N + \varepsilon(x)\pi(x) +N, &\textrm{if $\varepsilon(x) = -1$.}\\
\end{array}
\right.
\end{displaymath}
Summing over $H$ we get
\begin{equation*}
a\mathscr{h}_N  = N\mathscr{s}_N(a) + \kappa + N \eta(a, N),
\end{equation*}
where $\eta(a, N)=\varepsilon^{-1}(-1)$ and
\begin{equation*}
\kappa = \sum_{x\in H} \varepsilon(x)\pi(x) \equiv \mathscr{h}_N\pmod{2},
\end{equation*}
as $N$ is odd.
Therefore
\begin{equation*}
a \mathscr{h}_N  \equiv \mathscr{s}_N(a) +  \mathscr{h}_N +\eta(a, N)\pmod{2}.
\end{equation*}
In other words,
\begin{equation*}
 \eta(a, N)  \equiv  \mathscr{s}_N(a) + (a + 1)\mathscr{h}_N\pmod{2}.
\end{equation*}
But a result due to Schering~\cite[N. 14, p. 161]{schering:quad}
(cf. Cartier~\cite[p. 36]{cartier:symboles}) implies that
\begin{equation*}
(-1)^{\eta(a, N)} = 
\begin{cases}
    \legendre{a}{p} & \text{if $N=p^k$,} \\
        1         & \text{otherwise,}
    \end{cases}                                                                                                 
\end{equation*}
where $p$ is a prime number.
Hence Case 1 follows.

\bigskip
\noindent\textit{Case 2: Assume $N$ is even}.
Let
\begin{equation*}
\mathscr{S}(a, N) = \sum_{x = 1}^{N^\prime} \floor*{\frac{a x}{N}}.
\end{equation*}
Note that $a$ is odd.
Thus, Cohen~\cite[Corollary 2.2.14 (2), p. 42]{cohen:gtm239} implies that
\begin{align*}
 (-1)^{\mathscr{S}(a, N)} & =  (-1)^{\frac{1}{4}(N - 2)(a - 1)}\legendre{2N}{a} \\
                                      & =  (-1)^{\frac{1}{4}(N - 2)(a - 1)  +\frac{1}{8}(a^2 - 1) }\legendre{N}{a} \\
				      & =  (-1)^{\frac{1}{4}(N - 2)(a - 1)  +\frac{1}{8}(a^2 - 1) }\legendre{a}{N},
\end{align*}
where the last equality is an immediate consequence of our assumption that $N$ is a fundamental discriminant
and the quadratic reciprocity law.
Then,
by applying the multiplicative M\"obius inversion formula with respect to $N$,
we arrive at
\begin{equation*}
(-1)^{\mathscr{s}_N(a)} = (-1)^{\frac{\varphi(N)}{4}(a - 1)} \legendre{a}{N}^{\bm{\mu}}.
\end{equation*}
However, $N$ is an even discriminant.
Hence $4 \mid N$, and thus $4 \mid \varphi(N)$.
The latter, together with the fact that $a$ is odd, yield
\begin{equation*}
(-1)^{\frac{\varphi(N)}{4}(a - 1)} = 1
\end{equation*}
and the proposition follows.
\end{proof}

% {
% schering_symbol(a, m) =
% local(s);
%  s = 1;
% fordiv(m, d, s = s*kronecker(a, d)^moebius(m/d));
% return(s);
% }
% 
% m = 64*3;
% for(a = 1, m, if(gcd(a, m) ==1, print(a, " ",  schering_symbol(a, m), "   ", kronecker(a, 8))));

%%%%%%%%%%%%%%%%%%%%%%%%%%%%%%%%%%%%%%%%%%%%%%%%%%%%%%%%%%%%%%%%%%%%%
%            
%                 
%                   The modular units $f$ and $f^\mathscr{v}$
%                 
%                                    
%%%%%%%%%%%%%%%%%%%%%%%%%%%%%%%%%%%%%%%%%%%%%%%%%%%%%%%%%%%%%%%%%%%%%
\subsection{Proof of Theorem~\ref{thm:one}}\label{sub:diamonds}
The proof uses Proposition~\ref{prop:key} from the previous subsection,
together with
a result of Csirik~\cite{csirik:kernel}
and the divisibility property of $B_{2,\chi}$ mentioned in the introduction.

Following Kubert and Lang~\cite[p. 27]{lang:munits},
for each $z\in\CC$ and each lattice $L\subset\CC$ we may define the Klein form
\begin{equation*}
\mathfrak{k}(z, L) = e^{-\frac{1}{2}\eta(z,L)z}\sigma(z,L),
\end{equation*}
where $\sigma(z,L)$ is the Weierstra\ss\ sigma-function and $z\mapsto\eta(z,L)$
is the $\RR$-linear function that gives the quasi-periods of the Weierstra\ss\ zeta-function
with respect to the lattice $L$.
Put $\mathfrak{k}_a(\tau) = \mathfrak{k}(z,L_\tau)$,
where the point $a=(a_1,a_2)\in\RR^2$
is uniquely determined by $z = a_1\tau + a_2$ and $L_\tau = \ZZ\tau\oplus \ZZ$,
with $\tau$ lying in the upper-half plane
\begin{equation*}
\hyper = \{z\in\CC:\Im(z)>0\}.
\end{equation*}
%
% See more on Siegel units in:
%
% https://warwick.ac.uk/fac/sci/maths/people/staff/david_loeffler/teaching/modularcurves/lecture_notes.pdf
%
The \textit{Siegel function} is
\begin{equation*}
\mathfrak{g}_a (\tau) =\mathfrak{k}_a(\tau)\Delta^\frac{1}{12}(\tau),
\end{equation*}
where $\Delta^\frac{1}{12}(\tau)$ denotes
the principal part of the $12$-\textit{th} root of the discriminant $\Delta(\tau)$.
For each $r\in H$ and $\tau\in\hyper$ let
\begin{equation*}
g_r(\tau)=\mathfrak{g}_{(0,\frac{r}{N})}(\tau),
\end{equation*}
where $H$ is the half-system described in the previous subsection.
The modular unit $f$ is defined by the product
\begin{equation*}
f = \prod_{r\in H} g_r^{\chi(r)}.
\end{equation*}
Let $Y_1(N)$ be the moduli space that classifies ordered pairs $(E,P)$,
where $E$ is an elliptic curve over $\CC$ and $P$ is a point of order $N$ on $E$,
and let $Y_0(N)$ be the (coarse) moduli space that classifies ordered pairs that
classifies ordered pairs $(E, E^\prime)$ of elliptic curves $E$ and $E^\prime$ over $\CC$ together with
a cyclic isogeny $E\longrightarrow E^\prime$ of degree $N$.
Let $X_1(N)$ and $X_0(N)$ denote the corresponding smooth completions.
There are natural identifications
\begin{equation*}
X_1(N)(\CC)\cong \hyper^*/\Gamma_1(N)
\end{equation*}
and
\begin{equation*}
X_0(N)(\CC)\cong \hyper^*/\Gamma_0(N),
\end{equation*}
where
\begin{equation*}
\Gamma_1(N)=
\left\{\mtwo{a}{b}{c}{d}\in\Gamma_0(N)\,\vert\, a,d\equiv 1\pmod{N}\right\}
\end{equation*}
and
\begin{equation*}
\Gamma_0(N)=
\left\{\mtwo{a}{b}{c}{d}\in\textrm{SL}_2(\ZZ)\,\vert\,c\equiv 0\pmod{N}\right\}
\end{equation*}
act on the extended upper half plane $\hyper^*=\hyper\cup\PP^1(\QQ)$ via M\"obius transformations.
Following D.\thinspace{}E. Rohrlich and G.\thinspace{}F. Rohrlich\cite[p. 68]{rohrlich:1997},
for each $r\in\mathcal{G}$
the multiplication-by-$r$-map $P\mapsto [r]P$ induces an automorphism $\langle r\rangle$ of $X_1(N)$ over $\QQ$
known as the diamond automorphism
and $\langle r\rangle$ may be represented by any matrix $\abcd{a}{b}{c}{d}\in\Gamma_0(N)$ such that
\begin{equation*}
r\equiv d\pmod{N}.
\end{equation*}
So Theorem~\ref{thm:one} is equivalent to the assertion that
for each $M=\abcd{a}{b}{c}{d}\in\Gamma_0(N)$ we have
\begin{equation*}
f(M\tau) =  \chi^{\bm{\mu}}(d)  f(\tau)^{\chi(d)}.
\end{equation*}
To prove this, let $\{\cdot\}$ denote the lift map from $\mathcal{G}$ to $H$
and note Theorem 3.6 (p. 357) of Csirik~\cite[p. 357]{csirik:kernel} implies that
for each $r\in \mathcal{G}$ we have
\begin{equation*}
g_r(M\tau)=
\xi(M)\kappa(M, r)g_{\{r d\}}(\tau),
\end{equation*}
where
\begin{equation*}
\kappa(M, r) =
e^{-\pi i( \frac{r^2}{N^2}c d + \frac{r}{N}c + \lfloor\frac{r}{N}d\rfloor)},
\end{equation*}
and $\xi(M)$ is a $12$-th root of unity that depend only on $M$.
Hence
\begin{equation*}
\begin{split}
f(M\tau) &=
\prod_{r\in H} g_r(M\tau)^{\chi(r)}=
\prod_{r\in H}
\left(
\xi(M)
e^{\pi i(-\frac{r^2}{N^2}c d + 
\frac{r}{N}c + \lfloor \frac{r}{N}d \rfloor)}g_{\{rd\}}(\tau)
\right)^{\chi(r)} \\
                &= e^{\pi i S}f(\tau)^{\chi(d)},
\end{split}
\end{equation*}
where $S=S_1+S_2+S_3$ with
\begin{equation*}
\begin{split}
S_1 &= -\frac{c d}{N^2}\sum_{r\in H}\chi(r)r^2 = -c^\prime d\cdot\frac{1}{N}\sum_{r\in H}\chi(r)r^2,\\
S_2 &= \frac{c}{N}\sum_{r\in H}\chi(r) r = c^\prime\sum_{r\in H}\chi(r) r,\\
S_3 &= \sum_{r\in H}\chi(r)\left\lfloor \frac{r}{N} d \right\rfloor \equiv \mathscr{s}_N(d)\pmod{2}. \\
\end{split}
\end{equation*}
and $c^\prime = \tfrac{c}{N}$.
We divide into cases according to the parity of $c^\prime$:

\bigskip
\noindent\textit{Case 1: Assume $c^\prime$ is even}.
We immediately have
\begin{equation*}
S \equiv \mathscr{s}_N(d)\pmod{2}.
\end{equation*}
However, $M=\abcd{a}{b}{c}{d}\in\Gamma_0(N)$,
so $ad -c^\prime Nd = 1$, and thus $ad \equiv 1\pmod{2}$. Hence, $d$ is odd, and
Proposition~\ref{prop:key} implies that 
\begin{equation*}
(-1)^S = \legendre{d}{N}^{\bm{\mu}}
\end{equation*}
and Theorem~\ref{thm:one} follows.

\bigskip
\noindent\textit{Case 2: Assume $c^\prime$ is odd}.
The generalised Bernoulli numbers $B_{n,\chi}$,
attached to the character $\chi$ can be defined by the formal power series:
\begin{equation*}
\sum_{r=1}^N \chi(r)\frac{Xe^{rX}}{e^{NX}-1}=
\sum_{n=0}^\infty B_{n,\chi}\frac{X^n}{n!}
\end{equation*}
In particular,
\begin{equation*}
B_{2,\chi} = \frac{1}{N}\sum_{r=1}^N \chi(r)r^2 - \sum_{r=1}^N \chi(r)r.
\end{equation*}
As mentioned in the introduction, it is well-known that $B_{2,\chi}$ is divisible by $4$.
(Cf. Fox, Urbanowicz, and Williams~\cite[p. 819]{urbanowicz:progress}.)
In particular, $\tfrac{1}{2}B_{2,\chi}$ is even.
Because the character $\chi$ is even, we have
\begin{equation*}
\frac{1}{2}B_{2,\chi} = \frac{1}{N}\sum_{r=1}^{N^\prime} \chi(r)r^2 - \sum_{r=1}^{N^\prime} \chi(r)r.
\end{equation*}
So we arrive at the congruences
\begin{equation*}
\frac{1}{N}\sum_{r=1}^{N^\prime} \chi(r)r^2 \equiv \sum_{r=1}^{N^\prime} \chi(r)r\equiv\mathscr{h}_N \pmod{2}.
\end{equation*}
Therefore
\begin{equation*}
S \equiv  (d + 1)\mathscr{h}_N + \mathscr{s}_N(d)\pmod{2}.
\end{equation*}
Proposition~\ref{prop:key} implies that 
\begin{equation*}
(-1)^S = \legendre{d}{N}^{\bm{\mu}}
\end{equation*}
and Theorem~\ref{thm:one} follows.

%%%%%%%%%%%%%%%%%%%%%%%%%%%%%%%%%%%%%%%%%%%%%%%%%%%%%%%%%%%%%%%%%%%%%
%
%
%                    The $q$-expansion of $f$ and of $f^\mathscr{w}$
%   
%
%%%%%%%%%%%%%%%%%%%%%%%%%%%%%%%%%%%%%%%%%%%%%%%%%%%%%%%%%%%%%%%%%%%%%
\subsection{Proof of Theorem~\ref{thm:two}}
Kubert and Lang~\cite[p. 29]{lang:munits}
proved that the Siegel function $\mathfrak{g}_a(\tau)$ has a product expansion
\begin{equation*}
\mathfrak{g}_a(\tau)=-q_\tau^{\frac{1}{2}B_2(a_1)} e^{2\pi i a_2(a_1 - 1)/2} (1 - q_z)
\prod_{n=1}^\infty(1 - q_\tau^n q_z)(1 - q_\tau^n q_z^{-1}),
\end{equation*}
where $q_z=e^{2\pi iz}$ with $z\in\CC$.   
Let $\zeta_N=e^{2\pi i/N}$ and note that
\begin{align*}
f(\tau) &=  \prod_{r=1}^{N^\prime}
\left(\zeta_N^{-\frac{r}{2}}(1 - \zeta_N^r)
\prod_{n=1}^\infty(1 - q_\tau^n\zeta_N^r)(1 - q_\tau^n\zeta_N^{-r})
\right)^{\chi(r)} \\
        &= \left(
	         \prod_{r=1}^{N^\prime}
                 \zeta_N^{-\chi(r)\frac{r}{2}}(1 - \zeta_N^r)^{\chi(r)}
           \right)
	   \left(
\prod_{n=1}^\infty\prod_{r=1}^{N^\prime}(1 - q_\tau^n\zeta_N^r)^{\chi(r)}(1 - q_\tau^n\zeta_N^{-r})^{\chi(r)}
           \right) \\
        &= \kappa\prod_{n=1}^\infty\Psi(q^n),
\end{align*}
where 
\begin{align*}
\Psi(X) &=\prod_{r=1}^{N^\prime}(1-X\zeta_p^r)^{\chi(r)}(1-X\zeta_p^{-r})^{\chi(r)}\\
        &=\prod_{r=1}^{N^\prime}(1-X\zeta_p^r)^{\chi(r)}(1-X\zeta_p^{-r})^{\chi(-r)}\\
        &=\prod_{r=1}^{N-1}(1-X\zeta_p^r)^{\chi(r)} = 1 + O(X),
\end{align*}
and
\begin{equation*}
\kappa = \prod_{r=1}^{N^\prime}(\zeta_N^{-\frac{r}{2}} - \zeta_N^{\frac{r}{2}})^{\chi(r)} = \varepsilon_F^{-h_F}.
\end{equation*}
The last equality follows from the first equation in
Th\'eor\`eme~1 of Borevi\v c and \v Safarevi\v c~\cite[p.~385]{borevitch:nombres}, which states that:
\begin{equation*}
h_F = -\frac{1}{\log\varepsilon_F}\sum_{\substack{(r,N)=1 \\ 0<r<\frac{N}{2}}}\chi(r)\log\sin \frac{\pi r}{N}.
\end{equation*}
Hence
\begin{equation*}
f =\varepsilon_F^{-h_F} + O(q).
\end{equation*}
This proves the first assertion of Theorem~\ref{thm:two}.
To prove the second assertion,
consider that the Fricke involution $\mathscr{w}$ is induced by the action of
\begin{equation*}
\tau\mapsto -\frac{1}{N\tau},
\end{equation*}
which we may factor as the map $\tau\mapsto N\tau$ followed by the action of the matrix
\begin{equation*}
S=\mtwo{0}{-1}{1}{0}.
\end{equation*}
But \textbf{K0} and \textbf{K1} of Kubert and Lang~\cite[p. 27]{lang:munits}
say that for each $\alpha\in\textrm{SL}_2(\ZZ)$ and each $\tau\in\hyper$ we have
\begin{equation*}
\mathfrak{k}_{a\alpha}(\tau) = (c\tau + d)\mathfrak{k}_a(\alpha\tau).
\end{equation*}
Therefore
\begin{align*}
f^\mathscr{w}(\tau) &=  \prod_{r=1}^{N^\prime}\mathfrak{k}_{(-r/N, 0)}(N\tau)^{\chi(r)}\\
                     &=
\prod_{r=1}^{N^\prime}
\left(
q_\tau^{\frac{1}{2}B_2(\frac{r}{N})N}
(1-q_\tau^r)\prod_{n=1}^\infty (1 - q_\tau^{Nn + r})(1 - q_\tau^{Nn - r})
\right)^{\chi(r)}\\
                     &= q_\tau^{\frac{1}{4}B_{2,\chi}}
\prod_{r=1}^{N^\prime}(1-q_\tau^r)^{\chi(r)}
\prod_{r=1}^{N^\prime}
\prod_{n=1}^\infty (1-q_\tau^{Nn + r})^{\chi(r)}(1-q_\tau^{Nn - r})^{\chi(-r)}\\
                     &= q_\tau^{\frac{1}{4}B_{2,\chi}}\prod_{n=1}^\infty(1-q_\tau^n)^{\chi(n)}
\end{align*}
and Theorem~\ref{thm:two} follows.

%%%%%%%%%%%%%%%%%%%%%%%%%%%%%%%%%%%%%%%%%%%%%%%%%%%%%%%%%%%%%%%%%%%%%%%
%
%
%
%%%%%%%%%%%%%%%%%%%%%%%%%%%%%%%%%%%%%%%%%%%%%%%%%%%%%%%%%%%%%%%%%%%%%%%


\begin{thebibliography}{00}

\bibitem{borevitch:nombres}
Z.\thinspace{}I. {Borevi\v c} and I.\thinspace{}R. {\v Safarevi\v c},
  \emph{Th\'eorie des nombres}, Monographies internationales de math\'ematiques
  modernes, Gauthier-Villars, Paris, 1967.

\bibitem{cartier:symboles}
P.~Cartier, \emph{Sur une g\'en\'eralisation des symboles de
  {L}egendre-{J}acobi}, LEM \textbf{16} (1970), 31--48.

\bibitem{cohen:gtm239}
H.~Cohen, \emph{Number theory: Volume {I}: {T}ools and {D}iophantine
  equations}, Graduate Texts in Mathematics, vol. 239, Springer-Verlag, Berlin,
  2007.

\bibitem{csirik:kernel}
J.~Csirik, \emph{The {K}ernel of the {E}isenstein {I}deal}, Journal of Number
  Theory \textbf{92} (2002), no.~2, 348--375.

\bibitem{urbanowicz:progress}
G.\thinspace{}J. Fox, J.~Urbanowicz, and K.\thinspace{}S. Williams,
  \emph{Gauss' congruence from {D}irichlet's class number formula and
  generalizations}, Number Theory in Progress, vol. 1 and 2, Walter de Gruyter,
  Berlin, New York, 1999, Proceedings of the International Conference in Honor
  of the 60th Birthday of Andrzej Schinzel, Zakopane, Poland, June 30-July 9,
  1997, pp.~xvi + 1185.

\bibitem{lang:munits}
D.\thinspace{}S. Kubert and S.~Lang, \emph{Modular units}, Grundlehren der
  Mathematischen Wissenschaften, vol. 244, Springer-Verlag, New York, 1981.

\bibitem{leepark:level13}
Y.~Lee and Y.\thinspace{}K. Park, \emph{The level 13 analogue of the
  {R}ogers-{R}amanujan continued fraction and its modularity}, Journal of
  Number Theory \textbf{168} (2016), 306--333.

\bibitem{leopoldt:bernoulli}
H.\thinspace{}W. Leopoldt, \emph{Eine {V}erallgemeinerung der {B}ernoullischen
  {Z}ahlen}, Abh. Math. Sem. Univ. Hamburg \textbf{22} (1958), 131--140.

\bibitem{mazur:eisen}
B.~Mazur, \emph{Modular curves and the {E}isenstein ideal}, Inst. Hautes
  \'Etudes Sci. Publ. Math. (1977), no.~47, 33--186 (1978).

\bibitem{rohrlich:1997}
D.\thinspace{}E. Rohrlich and G.\thinspace{}F. Rohrlich, \emph{Modular
  {C}urves, {H}ecke {C}orrespondences, and {L}-{F}unctions}, Modular Forms and
  Fermat's Last Theorem (Gary Cornell, Joseph~H. Silverman, and Glenn Stevens,
  eds.), Springer, New York, 1997, pp.~41--100.

\bibitem{schering:quad}
E.~Schering, \emph{Zur {T}heorie der quadratischen {R}este}, Acta Math.
  \textbf{1} (1882), 153--170.
  
\bibitem{urbanowicz:b2}
J.~Urbanowicz, \emph{Connections between {$B_{2,\chi }$} for even quadratic {Dirichlet} characters $\chi $ and class numbers of appropriate imaginary quadratic fields, {I}},
Compositio Mathematica \textbf{75} (1990), no.~3, 247--270.

\bibitem{wwe:squarefree}
P.~Wake, P. and C. Wang-Erickson, \emph{The {E}isenstein ideal with squarefree level},
Advances in Mathematics \textbf{380} (2021), 247--270.

\bibitem{wiles:iwasawa}
A.~Wiles, \emph{The {I}wasawa conjecture for totally real fields}, Ann. of
  Math. \textbf{131} (1990), no.~3, 493--540.

\end{thebibliography}
\end{document}